\documentclass{article}
\usepackage[utf8]{inputenc}
\usepackage[english]{babel}
\usepackage{amsmath}
\usepackage{amsfonts}
\usepackage{amssymb}
\usepackage{amsthm}
\usepackage{graphicx}
\usepackage{array}
\usepackage[capitalize]{cleveref}
\usepackage{enumitem}
\usepackage{authblk}
\newcommand{\overbar}[1]{\mkern 1.5mu\overline{\mkern-1.5mu#1\mkern-1.5mu}\mkern 1.5mu}
\newcommand{\setdelim}{\; | \;}
\newcommand{\lspace}{L^2(X,m)}

\newcommand{\E}{\mathcal{E}}
\newcommand{\D}{\mathfrak{D}}

\usepackage{mathtools}

\crefname{property}{Property}{Properties}
\crefname{assum}{Assumption}{Assumptions}

\theoremstyle{plain}
\newtheorem{thm}{Theorem}[section]
\newtheorem{lem}[thm]{Lemma}
\newtheorem{corollary}[thm]{Corollary}

\theoremstyle{definition}
\newtheorem{defn}[thm]{Definition} 
\newtheorem{exmp}[thm]{Example} 
\newtheorem{remark}[thm]{Remark} 

\author{Burkhard Claus}

\affil{Department of Mathematics, TU Dresden}

\DeclareMathOperator{\normcap}{Cap_{\mathfrak{D}}}

\DeclareMathOperator{\supp}{supp}
\DeclareMathOperator{\sym}{sym}
\DeclareMathOperator{\epigraph}{epigraph}
\DeclareMathOperator{\convexhull}{Conv}
\DeclareMathOperator{\linspan}{span}
\DeclareMathOperator{\dom}{dom}

\title{Energy Spaces, Dirichlet Forms and Capacities in a Nonlinear Setting}
\date{July 2019}

\begin{document}

\maketitle

 \begin{abstract}
     In this article we study lower semicontinuous, convex functionals on real Hilbert spaces. In the first part of the article we construct a Banach space that serves as the energy space for such functionals. In the second part we study nonlinear Dirichlet forms, as defined by Cipriani and Grillo, and show, as it is well known in the bilinear case, that the energy space of such forms is a lattice. We define a capacity and introduce the notion quasicontinuity associated with these forms and prove several results, which are well known in the bilinear case.
 \end{abstract}

\section{Introduction}

The theory of gradients and subgradients of convex functions on a Hilbert space, as presented in \cite{Rockafellar_subgradients,Lions_subgradients} or \cite{Brezis_subgradients}, can be seen as a nonlinear counterpart to the theory of symmetric bilinear forms.

In the setting of Lions, one is given a Gelfand triple \( V \hookrightarrow H \hookrightarrow V'\) and a convex, differentiable function \(\E:V \rightarrow \mathbb{R}\). Then, the gradient is a nonlinear operator from \(V\) to \(V'\). In this situation, the Banach space \(V\) is sometimes called the energy space.

In the setting of Brezis, where \(\E : H \rightarrow [0,\infty]\) is merely a convex, lower semicontinuous function on the Hilbert space \(H\), such an energy space is not explicitly given. One goal of this article is to show that nevertheless a Banach space serving as an energy space can be constructed naturally, and therefore to partly unify the approaches of Lions and Brezis. In \cite{CM_metric_on_effective_domain} the authors equipped the effective domain of the functional with a metric but here we are interested in a linear space.

In the case of quadratic forms, this energy space is the domain of the bilinear form equipped with the usual Hilbert space structure, and in the general case the energy space is a Banach space. For example the energy space of the \(p\)-Laplace operator on an open domain \(X \subset \mathbb{R}^n\) with Neumann boundary conditions is the Sobolev space
\begin{align*}
    W^1_{p,2}(X)=\{ u \in \lspace \setdelim \nabla u \in L^p(X,m)\}.
\end{align*}

The second goal of this article is to study these energy spaces, when the functional is a nonlinear Dirichlet form.
These forms were introduced by Cipriani and Grillo in \cite{CG_Nonlinear_Dirichlet_Forms} as convex, lower semicontinuous functionals on an \(L^2\) space which generate order preserving and \(L^\infty\) contractive semigroups of nonlinear operators. Furthermore, they showed, using earlier results form Barthelemy \cite{Barthelemy_invariant_convex_sets} as well as Bénilan and Picard \cite{BC_completly_accretive_operators}, that this definition is equivalent to the intrinsic \cref{def:dirichlet_form} which we use in this paper.

In the theory of bilinear Dirichlet forms the Dirichlet spaces and the capacity are important building blocks of the vast theory of these forms and the semigroups as well as the Markov processes they generate \cite{FOT_Dirichlet_forms_Markov_Processes,BH_analysis_of_Dirichlet_forms,MR_nonsymmetric_Dirichlet_forms}.
We show that the Dirichlet space of a nonlinear Dirichlet form is a lattice and, under some assumptions, the lattice operation are continuous. We define a capacity and quasicontinuous functions and show that many of the results from the bilinear world still can be transferred to our setting.

In a forthcoming paper we want to use the results presented here to investigate boundary conditions and perturbations of Dirichlet forms.

We point out that in the article \cite{Biroli_strongly_local_nonlinear_dirichlet_forms,BV_strongly_local_nonlinear_dirichlet_forms} a capacity is defined in the nonlinear setting, too. The assumption \((H0)\) in both articles assumes that the Dirichlet space exists and that the functional is \(p\)-homogeneous. It turns out that both assumptions are not necessary to define a capacity and we give an explicit way to construct the Dirichlet space. In addition, non homogeneous examples like the energy of the \(\infty\)-Laplacian are also covered in our approach. Other examples include the \(p\) or \(p(x)\)-Laplacian on subsets of \(\mathbb{R}^d\) or arbitrary Riemannian manifolds, fractional versions of these operators and sums thereof.

\section{Energy Spaces of Symmetric Functionals}

In the following \(H\) is a real Hilbert space and \(\mathcal{E}: H \rightarrow [0,\infty]\) always denotes a convex and lower semicontinuous functional.
We define the effective domain
\begin{align*}
    \dom \E =\{ x \in H \setdelim \E(x) < \infty\},
\end{align*}
and call \(\E\)
\renewcommand{\labelenumi}{(\roman{enumi})}
\begin{enumerate}
 \item \underline{symmetric}, if \(\mathcal{E}(0)=0\) and \(
\mathcal{E}(-x)=\mathcal{E}(x)\)  for all \( x \in \lspace\),
\item \underline{quasilinear}, if \(\dom \E\) is a linear subspace of \(H\).
\end{enumerate}

\begin{remark}
For a convex functional \(\mathcal{E}: H \rightarrow [0,\infty]\), the condition \(
\mathcal{E}(-x)=\mathcal{E}(x)\)  already implies that \(0\) is a global minimizer of \(\mathcal{E}\). In the study of lower semicontinuous functionals in the context of partial differential equations and calculus of variations we are mostly interested in the minimizers and not the minimum itself. So without loss of generality we can assume that \(\E(0)=0\).
\end{remark}

We now construct a Banach space associated with the convex, lower semicontinuous functional \(\E\). This construction in based on the ideas of modular spaces used in the construction of Musielak-Orlicz space and variable exponent Lebesgue space. For reference of this procedure see \cite{Musielak_Modular_Spaces} or \cite{DHHR_Modular_Spaces}.

Let
\begin{align*}
    \mathcal{E}_1(x) \coloneqq \Vert x \Vert_H^2 + \mathcal{E}(x)
\end{align*}
and  \(\Vert \cdot \Vert_\mathfrak{D} : H \rightarrow [0,\infty]\) be defined as
\begin{align*}
\Vert x \Vert_\mathfrak{D} \coloneqq \inf\Big\{ \lambda>0 \setdelim \mathcal{E}_1\Big(\frac{x}{\lambda} \Big)\leq 1 \Big\}.
\end{align*}
We define the energy space of \(\E\) by
\begin{align*}
\mathfrak{D}& \coloneqq\{ x\in H \setdelim \exists \lambda>0 : \mathcal{E}_1(\lambda x)< \infty \}\\ &= \{x\in H \setdelim \Vert x \Vert_\D < \infty\}.
\end{align*}

\begin{thm}\label{dirichlet space is banach}
Let \(\mathcal{E}: H \rightarrow [0, \infty]\) be a symmetric, convex, lower semicontinuous functional on \(H\). Then the space \((\mathfrak{D},\Vert \cdot \Vert_\mathfrak{D})\) is a Banach space. In addition, the embedding \(i: \mathfrak{D} \rightarrow H\) is continuous and \(\Vert \cdot \Vert_\D\) is lower semicontinuous on \(H\).
\end{thm}
\begin{proof}
Observe that for all \(x \in H\) we have
\begin{align*}
\Vert x \Vert_H= \inf\Big\{ \lambda>0 : \Big\Vert\frac{x}{\lambda} \Big\Vert^2_H\leq 1\Big\} \leq \inf\Big\{ \lambda>0 : \Big\Vert\frac{x}{\lambda} \Big\Vert^2_H+\mathcal{E}\Big(\frac{x}{\lambda} \Big)\leq 1 \Big\} =\Vert x \Vert_\mathfrak{D}.
\end{align*}

Since \(\Vert . \Vert_H\) is a norm, this yields \(\Vert x \Vert_\mathfrak{D}=0\) if and only if \(x=0\). 

To show homogeneity, let \(\mu>0\). Then,
\begin{align*}
\Vert \mu x \Vert_\D&= \inf\Big\{ \lambda>0 \setdelim \E_1\Big(\frac{\mu x}{\lambda} \Big)\leq 1 \Big\} \\
    &= \inf\Big\{ \lambda\mu>0 \setdelim \E_1\Big(\frac{\mu x}{\mu\lambda} \Big)\leq 1 \Big\} \\
    &= \inf\Big\{ \lambda\mu>0 \setdelim \E_1\Big(\frac{x}{\lambda} \Big)\leq 1 \Big\}\\
    &= \mu \inf\Big\{ \lambda>0 \setdelim \E_1\Big(\frac{x}{\lambda} \Big)\leq 1 \Big\} \\
    &= \mu \Vert x \Vert_\D.
\end{align*}
Homogeneity for \(\mu <0\) follows, since \(\E_1(x)=\E_1(-x)\) which implies \(\Vert -x\Vert_\D= \Vert x\Vert_\D\). 

For the triangle inequality, let \(x,y \in \D\) and \(a> \Vert x\Vert_\D, b > \Vert y\Vert_\D \). Then
\begin{align*}
    \E_1\left(\frac{x+y}{a+b}\right)&=
    \E_1\left(\frac{a}{a+b}\frac{x}{a}+\frac{b}{a+b}\frac{y}{b}\right)\\
    &\leq \frac{a}{a+b}\E_1\left(\frac{x}{a}\right)+\frac{b}{a+b} \E_1\left(\frac{y}{b}\right) \\
    &\leq 1.
\end{align*} 
Hence 
\begin{align*}
    \Vert x+y\Vert_\D \leq a+b.
\end{align*}
taking the infimum over all possible \(a\) and \(b\) yields 
\begin{align*}
    \Vert x+y\Vert_\D \leq \Vert x\Vert_\D+\Vert y\Vert_\D.
\end{align*}
Thus \( \Vert  \cdot \Vert_\D\) is a norm. 
Since we already showed that 
\begin{align*}
    \Vert x \Vert_H \leq \Vert x \Vert_\D,
\end{align*} we know that the embedding \(\iota :\mathfrak{D} \rightarrow H\) is continuous.

It is easy to see, that \(\E_1(x) \leq 1\) if and only if \( \Vert x \Vert_\D \leq 1\). Thus,
\begin{align*}
    \{ x \in H \setdelim \E_1(x) \leq 1\}=\{ x \in H \setdelim \Vert x \Vert_\D \leq 1\}.
\end{align*}
Since \(\E_1\) is lower semicontinous, the set 
\begin{align*}
    B_1=\{ x \setdelim \Vert x \Vert_\D \leq 1\}
\end{align*}
is closed in \(H\). Let \(\mu \geq 0\). Then
\begin{align*}
   \{ x \setdelim \Vert x \Vert_\D \leq \mu\}= \mu B_1
\end{align*}
is also closed. Therefore \(\Vert \cdot \Vert_\D\) is lower semicontinuous on H.

Let \((x_n)_n\) be a Cauchy sequence in \(\D\). Then \((x_n)_n\) is a Cauchy sequence in \(H \). Since \(H\) is complete, there is an \(x \in H\) with \(x_n \rightarrow_H x \) and, since the norm \(\Vert \cdot \Vert_\D\) is lower semicontinuous on \(\lspace\) and Cauchy sequences are bounded,
\begin{align*}
    \Vert x \Vert_\D \leq \liminf_{m\rightarrow \infty} \Vert x_m \Vert_\D < \infty.
\end{align*}
Hence \(x \in \D\).
Let \(\epsilon >0\). We choose \(N \in \mathbb{N}\) such that for every \(n,m \geq N\)
\begin{align*}
    \Vert x_n-x_m\Vert_\D <\epsilon.
\end{align*}
Thus, the lower semicontinuity of the norm yields
\begin{align*}
    \Vert x_n-x \Vert_\D \leq \liminf_{m\rightarrow \infty} \Vert x_n-x_m \Vert_\D < \epsilon
\end{align*}
for every \(n\geq N\). Therefore, \(x_n \rightarrow_\D x\). Thus, \(\D\) is complete.
\end{proof}

Let \(\alpha >0\),
\begin{align*}
    \Vert x \Vert_{\D,\alpha} = \inf\Big\{ \lambda>0 : \E_1\Big(\frac{x}{\lambda} \Big)\leq \alpha \Big\}.
\end{align*}
and
\begin{align*}
    \vert x \vert_\D=\inf\Big\{ \lambda>0 : \E\Big(\frac{x}{\lambda} \Big)\leq 1 \Big\}.
\end{align*}
Similar arguments as in the first part of the previous proof show, that \(\Vert \cdot \Vert_{\D,\alpha}\) are norms for every \(\alpha>0\) and \(\vert \cdot \vert_\D\) is a seminorm.

\begin{thm}\label{thm:equivalent_norms}
For every \(\alpha >0 \), the norms \(\Vert \cdot \Vert_{\D}, \Vert \cdot \Vert_{\D,\alpha}\) and \(\Vert \cdot \Vert_H + \vert \cdot \vert_\D\) are equivalent norms on \(\D\).
\end{thm}
\begin{proof}
At first, we can use exactly the same proof as for \(\Vert \cdot \Vert_{\D}\) to show that \(\D\) together with \(\Vert \cdot \Vert_{\D,\alpha}\) is a Banach space.
Let \(0<\alpha_1<\alpha_2\). Then
\begin{align*}
    \Vert x \Vert_{\D,\alpha_2} \leq \Vert x \Vert_{\D,\alpha_1},
\end{align*}
for every \(x \in \D \), since \(\Vert x \Vert_{\D,\alpha_2}\) takes the infimum over a larger set. The open mapping theorem implies that comparable complete norms are equivalent. Thus, there is a \(C>0\) such that 
\begin{align*}
 \Vert x \Vert_{\D,\alpha_1}\leq C \Vert x \Vert_{\D,\alpha_2}
\end{align*}
for every \(x \in \D\).

For the second part, let \(x \in \D\). Obviously 
\begin{align*}
    \Vert x \Vert_H \leq \Vert x \Vert_\D
\end{align*}
and
\begin{align*}
    \vert x \vert_\D \leq \Vert x \Vert_\D.
\end{align*}
Hence 
\begin{align*}
    \Vert x \Vert_H + \vert x \vert_\D \leq 2 \Vert x \Vert_\D.
\end{align*}
On the other hand, let \(\lambda_1,\lambda_2\) such that \(\Vert x \Vert_H \leq \lambda_1\) and \(\vert x \vert_\D \leq \lambda_2\). Then
\begin{align*}
    \left\Vert \frac{x}{\lambda_1} \right \Vert_H \leq 1
\end{align*}
and 
\begin{align*}
    \E\left(\frac{x}{\lambda_2}\right) \leq 1.
\end{align*}
Hence, 
\begin{align*}
    \E_1\left(\frac{x}{\lambda_1+\lambda_2}\right) \leq 2.
\end{align*}
Thus,
\begin{align*}
    \Vert x \Vert_{\D,2} \leq \lambda_1+\lambda_2.
\end{align*}
Taking the infimum over all possible \(\lambda_1\) and \(\lambda_2\) yields
\begin{align*}
    \Vert x \Vert_{\D,2} \leq 
    \Vert x \Vert_H+\vert x \vert_\D.
\end{align*}
Since \(\Vert x \Vert_{\D,2}\) and \(\Vert x \Vert_{\D}\) are equivalent, there is a \(C>0\) such that
\begin{align*}
   C \Vert x \Vert_\D \leq  \Vert x \Vert_H+\vert x \vert_\D \leq 2 \Vert x \Vert_\D.
\end{align*}
Hence all the norms are equivalent.
\end{proof}

\begin{thm}
Let \(\E\) be a symmetric, convex and lower semicontinuous functional. Then \(\mathfrak{D}\) is a dual space.
\end{thm}
\begin{proof}
Kaijser proved in \cite{Kaijser_Dual_Space}, that a Banach space \(Y\) is a dual space if 
there is a set of continuous linear functionals \(E\) on \(Y\) that separates the points of \(Y\) and the closed unit ball of \(Y\) is compact in the weak topology generated by \(E\).

Let \(E=H'=H\). By the Hahn-Banach theorem, \(E\) separates the points of \(H\). Hence, it separates the points of \(\D\). On the other hand the closed unit ball \(B_1\) of \(\D\) is closed and bounded in \(H\). Therefore, it is compact in the weak* topology by Banach-Alaoglu. But weak and weak* topology coincide on reflexive spaces. Hence, \(B_1\) is compact in the weak topology induced by \(E\). Thus \(\D\) is a dual space.
\end{proof}
Note that Kaijser also showed that the predual is then given by the weak* closure of \(\linspan{E}\) in \(\D'\).

\begin{thm} \label{mod conv chara}
Let \((x_n)_{n\in \mathbb{N}}\) a sequence in \(\D\). Then \(x_n\rightarrow 0\) in \(\D\) if and only if for all \(\lambda >0:  \E_1(\lambda x_n) \rightarrow 0\).
\end{thm}
\begin{thm}[Norm-modular unit ball property] \label{unit ball prop}
Let \(x\in \D\). Then \(\Vert x \Vert_\D \leq 1\) if and only if \(\E_1(x)\leq 1\). In this case we have \(\E_1(x) \leq \Vert x \Vert_\D.\)
\end{thm}
\begin{proof}
For both theorems we omit the proof, see \cite[Lemma 2.1.9, 2.1.14 and Corollary 2.1.15]{DHHR_Modular_Spaces} instead. Note that the convexity and lower semicontinuity of \(\E\) and \(\E_1\) imply that \(\E_1\) is left continuous and therefore a modular in the notation of this reference.
\end{proof}

\section{Energy Spaces of non Symmetric Functionals}

\begin{defn}
Let \(\mathcal{E}: H \rightarrow [0, \infty]\) be a lower semicontinuous, convex functional on a real Hilbert space \(H\), such that \(\mathcal{E}(0)=0\). We define the symmetric closure \(\sym \E\) of \(\E\) by 
\begin{align*}
    \sym \mathcal{E} (u)=\sup\{ \mathcal{F}(u) \setdelim \mathcal{F} \text{ lsc, convex and } \mathcal{F}\leq \mathcal{E}_1(\cdot),\mathcal{E}_1(-\cdot)\}.
\end{align*}
\end{defn}

\begin{thm}\label{Thm:Symmetric_Hull_Domain}
Let \(\mathcal{E}: H \rightarrow [0, \infty]\) be a lower semicontinuous, convex functional on a Hilbert \(H\), such that \(\mathcal{E}(0)=0\). Then 
\begin{align*}
    \linspan{\dom(\mathcal{E})}= \linspan{\dom(\sym \mathcal{E})}.
\end{align*}
Moreover, any \(f \in \dom(\sym \mathcal{E})\) can be written as \(f=u-v\) for \(u,v \in \dom (\mathcal{E})\).
\end{thm}
\begin{proof}
We need to show that \(\dom(\sym \mathcal{E}) \subset  \linspan{\dom(\mathcal{E})}\).
For this purpose, let \(f \in \dom(\sym \mathcal{E})\). Then, there is a \(t\geq 0 \) such that \((f,t) \in \epigraph (\sym \mathcal{E})\). Since \( \sym \mathcal{E} \) is the lower semicontinuous closure of \[\convexhull( \epigraph \mathcal{E}_1(\cdot) \cup \epigraph(\mathcal{E}_1(-\cdot)) ),\] there is a sequence
\((f_n,t_n)_{n \in \mathbb{N}}\) in \(\convexhull( \epigraph(\mathcal{E}_1(\cdot) \cup \epigraph(\mathcal{E}_1(-\cdot)) )\) such that \( (f_n,t_n) \rightarrow (f,t) \) in \(H \times \mathbb{R}\). By definition of the convex hull, there are \( (u_n,s_n),(v_n,r_n) \in \epigraph \E_1\) and \(\lambda_n \in [0,1]\) such that
\begin{align*}
    (f_n,t_n)=\lambda_n(u_n,s_n)+ (1-\lambda_n)(-v_n,r_n).
\end{align*}
After choosing a subsequence, we may assume that \(\lambda_n \rightarrow \lambda \in [0,1]\). Now, let us assume \(\lambda \neq 0 \). Since \(r_n\) is positive, 
\begin{align*}
    0\leq  \lambda_n s_n \leq t_n \rightarrow t.
\end{align*}
Hence, \((s_n)_{n\in \mathbb{n}}\) is bounded since \(\lambda \neq 0\).
Again, after choosing a subsequence if necessary, there is a \(s \in \mathbb{R}\) such that \(s_n \rightarrow s\). By the definition of the epigraph, we know
\begin{align*}
    \mathcal{E}(u_n)+ \Vert u_n \Vert_H \leq s_n \rightarrow s. 
\end{align*}
Hence, \((u_n)_{n \in \mathbb{N}} \) is bounded in \(H\). Thus there is a subsequence, again denoted by \((u_n)_{n \in \mathbb{N}}\), such that \(u_n \rightharpoonup u\). Since \(\mathcal{E}\) is lower semicontinuous with respect to the weak topology on \(H\), \(u \in \dom \mathcal{E}\). On the other hand,
\begin{align*}
    f_n= \lambda_n u_n - (1- \lambda_n) v_n,
\end{align*}
and \(f_n \rightharpoonup f,u_n \rightharpoonup u\) implies \((1-\lambda_n)v_n \rightharpoonup v\). Hence,
\begin{align*}
    \mathcal{E}(v) &\leq \liminf_{n \rightarrow \infty} \mathcal{E}\big((1-\lambda_n)v_n \big) \\
    &\leq \liminf_{n \rightarrow \infty} (1-\lambda_n) \mathcal{E}(v_n) \\
    &\leq \liminf_{n \rightarrow \infty} (1-\lambda_n) r_n \\
    &=  \liminf_{n \rightarrow \infty} t_n- \lambda_n s_n \\
    &= t- \lambda s \\& < \infty
\end{align*}
Thus, \(v \in \dom(\mathcal{E})\) and \(f=u-v\). Hence, \(f \in \linspan \dom( \mathcal{E})\).

If \(\lambda=0\) then \(1-\lambda \neq 0\) and switching \(r\) and \(s\) as well as \(u\) and \(v\) in the previous argument yields the conclusion.
\end{proof}

\begin{defn}
Let \(\mathcal{E}: H \rightarrow [0, \infty]\) be a lower semicontinuous, convex functional on a Hilbert \(H\), such that \(\mathcal{E}(0)=0\). We call the space \(\D\) associated with \(\sym\E\) the \underline{energy space} of \(\E\).
\end{defn}

\begin{remark}
If \(\E \) is already symmetric, then \(\sym \E =\E_1\) and both definitions of \(\D\) coincide.
\end{remark}

\begin{remark}
In this article we always start with a Hilbert space H. For the results up to this point this is in fact not necessary. Similar results hold in the case of reflexive Banach spaces like for example \(L^p\) for \(1<p< \infty\). It might be more natural in some applications, for example if \(\E\) itself is \(p\)-homogeneous, to define
\begin{align*}
    \E_1(\cdot)=\Vert \cdot \Vert_{L^p}^p + \E(\cdot),
\end{align*}
which leads again to a \(p\)-homogeneous functional.
\end{remark}

\begin{remark}
If \(0\) is not a global minimizer of \(\E\), but if nevertheless \(\E\) possesses a global minimizer, then we can shift the functional \(\E\) and still define an energy space. This shifted functional is still a Dirichlet form if \(\E\) is a Dirichlet form. But one looses the property \(\dom \E \subset \mathfrak{D}\) if one defines \(\D\) by using the shifted functional.
\end{remark}

\begin{remark}
We also assume that \(\E\) is convex. But in fact we only need that \(\E_1\) is convex, or more general that
\begin{align*}
    \E_\alpha (x) = \Vert x \Vert_H + \alpha \E(x) 
\end{align*}
is convex for some \(\alpha>0\). Then \(\E_\alpha\) generates a energy space and the norms are equivalent for any \(\alpha\) such that \(\E_\alpha\) is convex. To prove this, let \(\alpha < \beta \) such that \(\E_\alpha\) and \(\E_\beta\) are convex. Then \(\E_\alpha\) resp. \(\E_\beta\) generate an energy space \(\D_\alpha\) resp. \(\D_\beta\). By \cref{thm:equivalent_norms}, the norms \(\Vert \cdot \Vert_{\D_\alpha}\) and \(\Vert \cdot \Vert_{\D_\beta}\) are equivalent to \(\Vert \cdot \Vert_H +\vert \cdot \vert_{\D_\alpha}\) and \(\Vert \cdot \Vert_H+\vert \cdot \vert_{\D_\beta}\). Obviously
\begin{align*}
    \vert x \vert_{\D_\alpha} \leq \vert x \vert_{\D_\beta}
\end{align*}
by definition. On the other hand,
\begin{align*}
    \vert x \vert_{\D_\beta} &= \inf\Big\{ \lambda>0 : \beta \E\Big(\frac{x}{\lambda} \Big)\leq 1 \Big\} \\
    &= \inf\Big\{ \lambda>0 : \alpha \frac{\beta }{\alpha}\E\Big(\frac{x}{\lambda} \Big)\leq 1 \Big\} \\
    & \leq \inf\Big\{ \lambda>0 : \alpha\E\Big(\frac{\beta x}{\alpha \lambda} \Big)\leq 1 \Big\} \\
    &= \inf\Big\{ \frac{\beta }{ \alpha }\lambda>0 :  \alpha \E\Big(\frac{x}{\lambda} \Big)\leq 1 \Big\} \\
    &= \frac{\beta  }{ \alpha } \vert x \vert_ {\D_\alpha}.
\end{align*}
Hence, the norms on \(\D_\alpha\) and \(\D_\beta\) are equivalent. Thus the energy space does not depend on the choice of \(\alpha\) and we can define an energy space for any functional \(\E\) such that
\begin{align*}
    \E_\alpha (x) = \Vert x \Vert_H + \alpha \E(x) 
\end{align*}
is convex. Such functionals \(\E\) are called \(\omega\)-semiconvex, where \(\omega=\frac{1}{\alpha}\).

Taking the previous remarks into account we can define an energy space for any lower semicontinuous, semiconvex functional \(\E\) on any reflexive Banach space.
\end{remark}

\section{Dirichlet Forms and Dirichlet Spaces}

Let \(X\) be a countably generated Borel space and \(m\) a \(\sigma\)-finite Borel measure such that \(\supp(m)=X\). The following definition was first introduced by \cite{CG_Nonlinear_Dirichlet_Forms}.

\begin{defn} \label{def:dirichlet_form}
Let \(\mathcal{E} : \lspace \rightarrow [0,\infty]\) be a convex and lower semicontinuous functional with dense effective domain. We call \(\mathcal{E}\) a \underline{Dirichlet form} if
\begin{align}\label[property]{diricheqn1}
&\mathcal{E}(u \wedge v) + \mathcal{E}(u \vee v) \leq \mathcal{E}(u)+\mathcal{E}(v)
\end{align}
and
\begin{align} \label[property]{diricheqn2}
&\mathcal{E}\bigg(v+\frac{1}{2}\big( (u-v+\alpha)_+-(u-v-\alpha)_- \big)\bigg) \nonumber \\&\quad \quad +\mathcal{E}\bigg(u-\frac{1}{2}\big( (u-v+\alpha)_+-(u-v-\alpha)_- \big)\bigg) \leq \mathcal{E}(u)+\mathcal{E}(v)
\end{align}
for every \(u,v \in \lspace,\alpha>0\).

If \(\E(0)=0\) we call the energy space \(\D\) of \(\E\) the \underline{Dirichlet space}.
\end{defn}

\begin{remark}
It is well known, that a convex, lower semicontinuous functional with dense domain generates a semigroup of nonlinear contractions. In the theory of bilinear forms, Dirichlet forms are exactly those forms, which generate sub-Markovian semigroups, that is order-preserving and \(L^\infty\)-contractive semigroups. The two conditions in the previous definition hold if and only if the semigroup generated by \(\E\) is sub-Markovian. 
\end{remark}

\begin{remark}
If \(a \) is a bilinear Dirichlet form, then 
\begin{align*}
    \E(u)=\begin{cases} \frac{1}{2} a(u,u) &\text{ if } x \in D(a) \\
    \infty &\text{ otherwise}
    \end{cases}
\end{align*}
is a Dirichlet form in the sense of the previous definition and the Dirichlet space \(\D\) and the classical Dirichlet space, that is the form domain \(D(a)\) together with the scalar product \(a_1\), are isomorphic as topological vector spaces. Furthermore, there is a scalar product on \(\D\) which induces the norm \(\Vert \cdot \Vert_\D\).
\end{remark}

\begin{defn}
We call \(p \in W^{1, \infty}(\mathbb{R})\) a \underline{normal contraction} if \(0 \leq p' \leq 1\) and \(p(0)=0\).
\end{defn}

\begin{remark}
Note that any such normal contraction \(p\) induces an operator \(T_p: \lspace \rightarrow \lspace\) defined by
\begin{align*}
    (T_p u )(x)=p(u(x)).
\end{align*}
We do not distinguish between \(p\) and \(T_p\) and denote both objects by \(p\).
\end{remark}

\begin{thm}[Beurling-Deny Criterion]\label{thm:beuerling:deny_criterion}

Let \(\E: \lspace \rightarrow [0,\infty] \) be a convex, lower semicontinuous functional. Then \(\E\) is a Dirichlet form if and only if 
\begin{align*}
    \E(u- p(u-v)) + \E(v+ p(u-v)) \leq \E(u)+\E(v)
\end{align*}
for every \(u,v  \in \lspace\) and every normal contraction \(p\).
\end{thm}

The proof follows from \cite[Proposition 1.2, Lemma 7.1, Proposition 7.2]{BC_completly_accretive_operators} and  \cite[Corollary 2.1]{BP_maximum_principle_nonlinear}.

\begin{thm} \label{Thm:D_Riesz_Subspace}
Let \(\mathcal{E}\) be a Dirichlet form on \(\lspace\) such that \(\E(0)=0\). Then the space \(\mathfrak{D}\) is a Riesz subspace of \(\lspace\). That is, for \(u,v \in \D\) we have
\begin{align*}
u \wedge v, u \vee v \in \mathfrak{D}.
\end{align*}
\end{thm}
\begin{proof}
Let \(g \in \mathfrak{D}\). Since 
\[\mathfrak{D}= \big\{ f \in \lspace \setdelim \text{ there is } \lambda>0 \text{ such that }\lambda f \in \dom(\sym\E) \big\},\]
there is a \(\lambda \geq 0\) and a function \(f \in \dom(\sym \E)\) such that \(g=\lambda f\). We showed in \cref{Thm:Symmetric_Hull_Domain}, that there are \(u,v \in \dom(\mathcal{E})\) such that \(f=u-v\). Hence,
\begin{align*}
    g \vee 0&= (\lambda u - \lambda v )\vee 0 \\
    &=\lambda (u-v) \vee 0 \\
    &=\lambda( u \vee v) -\lambda v.
\end{align*}
Since \(u \vee v, v \in \dom(\mathcal{E})\), \(g \vee 0 \in \linspan \dom(\mathcal{E}) = \mathfrak{D}\). The same argument works for \(g \wedge 0\). This implies the claim, since
\begin{align*}
    x \wedge y = (x-y) \wedge 0 +y.
\end{align*}
\end{proof}

\begin{thm} \label{Thm:D_Riesz_Subspace_Norm}
Let \(\mathcal{E}\) be a symmetric Dirichlet form on \(\lspace\). Then
\begin{align*}
\Vert u \wedge v\Vert_\D \leq \Vert u\Vert_\D+\Vert v\Vert_\D
\end{align*}
for every \(u,v \in \D\).
\end{thm}

\begin{proof}
Let us assume the contrary. Then, there are \(u,v \in \mathfrak{D} \) such that 
\begin{align*}
 \Vert u \wedge v \Vert_\D > \Vert u\Vert_\D+\Vert v\Vert_\D.
\end{align*}
Note that, either \(u\) or \(v\) is non zero.
Therefore, we can choose a  \( \lambda \in \left(\frac{1}{\Vert u \wedge v \Vert_\D},\frac{1}{\Vert u\Vert_\D+\Vert v\Vert_\D} \right) \). This implies \(\lambda\Vert u \wedge v \Vert_\D >1\) and \(\Vert \lambda u \Vert_\D,\Vert \lambda v\Vert_\D <1\). Hence, by \cref{unit ball prop},
\begin{align*}
\mathcal{E}_1\big(\lambda(u \wedge v)\big) \geq \lambda \Vert u \wedge v \Vert_\D > \Vert \lambda u\Vert_\D+\Vert \lambda v\Vert_\D \geq \mathcal{E}_1\big(\lambda u \big)+\mathcal{E}_1\big(\lambda v\big).
\end{align*}
Since \(\mathcal{E}_1\) is a Dirichlet form, the previous inequality is a contradiction.

\end{proof}

One could hope that the following even stronger inequality holds
\begin{align*}
    \Vert u \wedge v \Vert_\D +\Vert u \vee v \Vert_\D \leq \Vert u \Vert_\D + \Vert v \Vert_\D .
\end{align*}
But this is false, as the next example illustrates.

\begin{exmp}
Let \(X= [0,1]\) and \(m\) the Lebesgue measure. We consider the functional \(\mathcal{E}: \lspace \rightarrow [0,\infty]\) given by
\begin{align*}
    \mathcal{E}(u)=\int_X \chi_{[0,1]}(\vert u(x) \vert ) dx,
\end{align*}
where \(\chi_A: \mathbb{R} \rightarrow \{0,\infty\}\) is given by
\begin{align*}
    \chi_A(x)=\begin{cases} 0 & \text{ if } x \in A, \\
    \infty & \text{ else. }
    \end{cases}
\end{align*}
The functional \(\mathcal{E}\) is lower semicontinuous, by Fatou's Lemma. Additionally, it is convex, since \(\chi_{[0,1]}\) is convex and it is easy to verify that \(\mathcal{E}\) satisfies the projection inequalities of \cref{def:dirichlet_form}. Hence, it is a symmetric Dirichlet form. The norm \(\Vert \cdot \Vert_\D\) is induced by the functional
\begin{align*}
    \mathcal{E}_1(u)=\Vert u \Vert_2^2 + \mathcal{E}(u).
\end{align*}
By definition, the Dirichlet space \(\mathfrak{D}\) is given by all elements \(u\in \lspace\) such that there is a \(\lambda>0\) satisfying 
\begin{align*}
    \mathcal{E}(\lambda^{-1} u)=\int_X \chi_{[0,1]}(\vert \lambda^{-1} u(x) \vert ) dx \leq 1,
\end{align*}
which is equivalent to 
\begin{align*}
    \chi_{[0,1]}(\vert \lambda^{-1} u(x) \vert )=0
\end{align*}
almost everywhere, or in other words
\begin{align*}
  \vert  u(x) \vert \leq \lambda \text{ for almost all } x \in X
\end{align*}
Thus, \(u \in \mathfrak{D}\) if and only if \(u \in L^\infty(X,m)\). Analogously the norm is given by 
\begin{align*}
    \Vert u \Vert_\D= \max\{\Vert u \Vert_2,\Vert u \Vert_\infty\}.
\end{align*}
Since \(([0,1],m)\) is a probability space, the norm \(\Vert \cdot \Vert_\D\) is equal to \(\Vert \cdot \Vert_\infty\).
Now let 
\begin{align*}
        f(x)=\begin{cases} -1 & \text{ if } x< \frac{1}{2}, \\
    1 & \text{ if } x\geq \frac{1}{2}.
    \end{cases}
\end{align*}
Then, \(f \in \mathfrak{D}\). But
\begin{align*}
     \Vert f \Vert_\D=\Vert f \wedge 0 \Vert_\D=1=\Vert f \vee 0 \Vert_\D.
\end{align*}
Thus, 
\begin{align*}
    \Vert f \wedge 0 \Vert_\D + \Vert f \vee 0 \Vert_\D =2 > 1 = \Vert f \Vert_\D + \Vert 0 \Vert_\D.
\end{align*}
Therefore, the inequality 
\begin{align*}
    \Vert u \wedge v \Vert_\D + \Vert u \vee v \Vert_\D \leq \Vert u \Vert +\Vert v \Vert,
\end{align*}
does not hold for the energy norm of a Dirichlet form in general.
\end{exmp}

\begin{thm}\label{inf sup sontinuity}
Let \(\mathcal{E}\) be a quasilinear, symmetric Dirichlet form and \(v \in \mathfrak{D}\).
Then the lattice operations are continuous.
\end{thm}
\begin{proof}
 
 At first, let \(v \in \mathfrak{D}\) such that \(v\geq 0\), let \((u_n)_{n\in \mathbb{N}}\) be a sequence in \(\mathfrak{D}\) converging to \(0\) and \(\lambda>0\). Then \(u_n \wedge v \rightarrow 0\) and \( u_n \vee v \rightarrow v\) in \(\lspace\) for \(n \rightarrow \infty\). Hence, the lower semicontinuity of \(\mathcal{E}\) implies
 \begin{align*}
     \mathcal{E}_1(0)+\mathcal{E}_1(\lambda v)&\leq \liminf_{n\rightarrow \infty} \mathcal{E}_1\big(\lambda(u_n \wedge v)\big)+\mathcal{E}_1\big(\lambda (u_n \vee v)\big) \\ &\leq 
     \limsup_{n\rightarrow \infty} \mathcal{E}_1\big(\lambda(u_n \wedge v)\big)+\mathcal{E}_1\big(\lambda (u_n \vee v)\big)
     \\ &\leq \limsup_{n\rightarrow \infty} \mathcal{E}_1\big(\lambda u_n\big)+\mathcal{E}_1\big(\lambda  v\big),
 \end{align*}
 where we used lower semicontinuity and the first inequality from \cref{def:dirichlet_form}. Since \(\E_1\) is quasilinear, \(\D= \dom \E_1\), and
 since \(\mathcal{E}_1\) is convex and lower semicontinuous on \(\D\), it is continuous on \(\D\). Therefore, since \(\E_1(0)=0\) the previous inequality yields
 \begin{align*}
     \mathcal{E}_1(\lambda v)&= \lim_{n\rightarrow \infty} \mathcal{E}_1\big(\lambda(u_n \wedge v)\big)+\mathcal{E}_1\big(\lambda (u_n \vee v)\big) 
 \end{align*}
 We already know, that
 \begin{align*}
     \mathcal{E}_1(\lambda v)= \mathcal{E}_1(\lambda(v \vee 0)) \leq \liminf_{n\rightarrow \infty} \mathcal{E}_1\big(\lambda (u_n \vee v)\big).
 \end{align*}
 Thus, \(\lim_{n\rightarrow \infty} \mathcal{E}_1\big(\lambda(u_n \wedge v) \big) =0\) and \cref{mod conv chara} implies \(u_n \wedge v \rightarrow 0\) in \(\mathfrak{D}\). In addition, \(v=v \wedge u_n + v \vee u_n -u_n\) which implies
 \begin{align*}
     \lim_{n\rightarrow \infty}v \vee u_n=\lim_{n\rightarrow \infty} v + u_n - v \wedge u_n = v+ 0 +0.
 \end{align*}
 Hence, \(v \vee u_n \rightarrow v \vee 0\).
 
 Now, let \(v \in \mathfrak{D}\) be arbitrary. Since \( \Vert u^+_n \Vert_\D \leq \Vert u_n \Vert_\D \), we know that \(u^+_n , u^-_n \rightarrow 0\). This implies
 \begin{align*}
     v \wedge u_n= v^+ \wedge u_n^+ - v^- \vee u_n^-.
 \end{align*}
 We now use the first part again, which yields 
 \begin{align*}
     \lim_{n\rightarrow \infty}v \wedge u_n&=\lim_{n\rightarrow \infty} v^+ \wedge u_n^+ - v^- \vee u_n^-\\
     &= v^+\wedge0 - v^- \vee 0 \\
     &= 0 + v \wedge 0.
 \end{align*}
 Finally, let \((u_n)_n\) be an arbitrary sequence in \(\mathfrak{D}\) converging to \(u\). Then \(u_n-u\) converges to 0 and
 \begin{align*}
     \lim_{n\rightarrow \infty}  u_n\wedge v =\lim_{n\rightarrow \infty} (u_n-u)\wedge (v-u) + u =0 \wedge (v-u) + u= u \wedge v.
 \end{align*}
Hence, the infimum is separately continuous.
Now, let \((u_n)_{n\in \mathbb{N}}\) and \((v_n)_{n\in \mathbb{N}}\) in \(\mathfrak{D}\) converging to \(u\) and \(v\). Then
\begin{align*}
    v_n \wedge u_n = 0 \wedge (u_n-v_n) + v_n.
\end{align*}
Hence,
\begin{align*}
    \lim_{n  \rightarrow \mathbb{N}} v_n \wedge u_n= 0 \wedge (u-v) + v = v \wedge u.
\end{align*}
The claim for the supremum follows by using the identity
\begin{align*}
    u + v = u \wedge v + u \vee v.
\end{align*}
\end{proof}

\begin{lem}\label{cutof continuity}
 If \(\mathcal{E}\) is a symmetric Dirichlet form, then 
\[ \mathcal{E}\big( -c \vee u \wedge c \big)  \leq \mathcal{E}\big(u\big)\]
for every \(u \in \mathfrak{D},c\geq 0\) and, if \(\mathcal{E}\) is quasilinear, then
\begin{align*}
    \lim_{n\rightarrow \infty} -n \vee u \wedge n = u \text{ in } \D.
\end{align*}
\end{lem}

\begin{proof}

The map \(p: \mathbb{R} \rightarrow \mathbb{R}\) given by \(p(x)=(-n)\vee x \wedge n \) is a normal contraction. Hence, \cref{thm:beuerling:deny_criterion} implies 
\begin{align*}
    \E_1(u - p(u-v)) + \E_1(v+p(u-v)) \leq \E_1(u) + \E_1(v).
\end{align*}
Setting \(v=0\) implies
\begin{align*}
    \E_1((-n)\vee u \wedge n) \leq \E_1 (u),
\end{align*}
since \(\E_1(0)=0\) and \(\E_1 \geq 0\).

For the second part, let us assume that \(\mathcal{E}\) is quasilinear. By Property (\ref{diricheqn2}) in the definition of Dirichlet forms, we have 
\begin{align*}
&\mathcal{E}_1\bigg(v+\frac{1}{2}\big( (u-v+\alpha)_+-(u-v-\alpha)_- \big)\bigg) \\&\quad \quad+\mathcal{E}_1\bigg(u-\frac{1}{2}\big( (u-v+\alpha)_+-(u-v-\alpha)_- \big)\bigg) \leq \mathcal{E}_1(u)+\mathcal{E}_1(v)
\end{align*}
for every \(u,v \in \lspace,\alpha>0\). Plugging in \(v=0\) and replaying \(\alpha\) by \(n\) yields
\begin{align}\label{cutof continuity:eqn1}
    \mathcal{E}_1\bigg( \frac{1}{2}\big(u+ (-n) \vee u \wedge n \big)\bigg) +\mathcal{E}_1\bigg(\frac{1}{2}\big(u- (-n) \vee u \wedge n \big)\bigg) \leq \mathcal{E}_1(u).
\end{align}
Let \(\lambda >0\).
Since \(\lim_{n\rightarrow \infty} -n \vee u \wedge n = u\) in \( \lspace\), the lower semicontinuity of \( \mathcal{E}_1\) and the previous inequality imply
\begin{align*}
    \mathcal{E}_1(0)+\mathcal{E}_1(\lambda u) &\leq \liminf_{n\rightarrow \infty} \mathcal{E}_1\bigg( \frac{\lambda}{2}\big(u+ (-n) \vee u \wedge n \big)\bigg) +\mathcal{E}_1\bigg(\frac{\lambda}{2}\big(u- (-n) \vee u \wedge n \big)\bigg)\\
    &\leq \limsup_{n \rightarrow \infty} \mathcal{E}_1\bigg( \frac{\lambda}{2}\big(u+ (-n) \vee u \wedge n \big)\bigg) +\mathcal{E}_1\bigg(\frac{\lambda}{2}\big(u- (-n) \vee u \wedge n \big)\bigg) \\
    &\leq \mathcal{E}_1(\lambda u),
\end{align*}
Since \(\E_1(0)=0\),
\begin{align}\label{cutof continuity:eqn2}
    \lim_{n\rightarrow \infty} \mathcal{E}_1\bigg( \frac{\lambda}{2}\big(u+ (-n) \vee u \wedge n \big)\bigg) +\mathcal{E}_1\bigg(\frac{\lambda}{2}\big(u- (-n) \vee u \wedge n \big)\bigg) = \mathcal{E}_1(\lambda u).
\end{align}

Since \(\mathcal{E}_1\) is lower semicontinuous, 
\begin{align*}
    \mathcal{E}_1(\lambda u) \leq \liminf_{n\rightarrow \infty} \mathcal{E}_1\bigg( \frac{\lambda}{2}\big(u+ (-n) \vee u \wedge n \big)\bigg).
\end{align*}

Since \(\E_1\) is quasilinear, \(\mathcal{E}_1(\lambda u)\) is finite and 
\begin{align*}
    0&\leq \limsup_{n\rightarrow \infty} \mathcal{E}_1\bigg( \frac{\lambda}{2}\big(u- (-n) \vee u \wedge n \big)\bigg) \\
    &\leq \limsup_{n\rightarrow \infty} \mathcal{E}_1\bigg( \frac{\lambda}{2}\big(u- (-n) \vee u \wedge n \big)\bigg)+
    \mathcal{E}_1\bigg( \frac{\lambda}{2}\big(u+ (-n) \vee u \wedge n \big)\bigg) 
    \\&\;+  \limsup_{n \rightarrow \infty} -\mathcal{E}_1\bigg( \frac{\lambda}{2}\big(u+ (-n) \vee u \wedge n \big)\bigg) 
    \\ & \leq \E_1(\lambda u) - \E_1(\lambda u) \\ &=0
\end{align*}

Hence,
\begin{align*}
    \lim_{n\rightarrow \infty} \mathcal{E}_1\bigg(\frac{\lambda}{2}\big(u- (-n) \vee u \wedge n \big)\bigg) = 0.
\end{align*}
\Cref{mod conv chara} implies \(\frac{1}{2}\big(u- (-n) \vee u \wedge n \big) \rightarrow 0\) for \( n\rightarrow \infty\) in \(\mathfrak{D}\). Since \(u\) was arbitrary, this yields \((-n) \vee u \wedge n \rightarrow u\) in \(\mathfrak{D}\) for every \(u \in \mathfrak{D}\) and \(n \rightarrow \infty\).
\end{proof}

\begin{corollary}
Let \(u,g \in \mathfrak{D}\). Then \(u \wedge ( c-g) \in \mathfrak{D}\) for every \(c>0\).
\end{corollary}
\begin{proof}
\begin{align*}
    u \wedge ( c-g)= \big( (u+g)\wedge c \big)-g
\end{align*}
\end{proof}

We give an example of a Dirichlet form where the lattice operations are not continuous on the associated Dirichlet space.

\begin{exmp}
    Let \(X=[0,1]\) with the usual Lebesgue measure. Let
    \begin{align*}
        \E(u)= \begin{cases} 0 &\text{ if \(u'\) exists and \(\vert u' \vert \leq 1 \) a.e. } \\
        \infty &\text{else}.
        \end{cases}
    \end{align*}
    Let us show that \(\E\) is lower semicontinuous. Let \((u_n)_{n \in \mathbb{N}}\) be a sequence in \(\lspace\) converging to \(u\) such that \(\E(u_n)=0\) for every \(n \in \mathbb{N}\). Then, \(\Vert u'_n \Vert_{L^\infty} \leq 1 \) and there exists \(v \in L^\infty(X), \Vert v \Vert_\infty \leq 1\) such that, up to a subsequence, \(u_n' \rightarrow v \) in the weak* topology of \(L^{\infty}(X)\). For every \( \phi \in C_c^\infty(X) \)
    \begin{align*}
        \int_X u_n' \phi = - \int_X u_n \phi'.
    \end{align*}
    Passing to the limit yields 
       \begin{align*}
        \int_X v \phi = - \int_X u \phi'.
    \end{align*} 
    Therefore \(u'=v\) and \(\E(u)=0\). 
    
    It is easy to see that \(\E\) is a Dirichlet form. The Dirichlet space \(\D\) coincides with \(W^{1,\infty}(X)\). Let \(f(x)=x\) and \(g_n(x)= \frac{1}{n}\). Then \(g_n\rightarrow 0 \) in \(\D\) but 
    \begin{align*}
        \Vert f \wedge g_n \Vert_{\D} \geq 1
    \end{align*}
    for every \(n\in \mathbb{N}\). Hence, \(f \wedge g_n \) does not converge to \(0\) in \(\D\).
\end{exmp}

\section{Capacity}

One tool in the study of classical Dirichlet forms is the capacity. For a reference in the classical case, see for example \cite{FOT_Dirichlet_forms_Markov_Processes}. From now on let us assume that \(X\) is a topological measure space and \(\supp m =X\).

In this section \(\mathcal{E}\) denotes a symmetric Dirichlet form and \((\mathfrak{D},\Vert.\Vert_\D)\) the associated Dirichlet space.

\begin{defn}

Let \(A \subset X \). We define the set
\begin{align*}
    \mathcal{L}_A=\{u \in \lspace \setdelim u\geq 1 \text{ on } U, A \subset U, U \text{ open } \},
\end{align*}
and the \underline{(norm-)capacity} by
\[\normcap(A)=\inf\{\Vert u \Vert_\D \setdelim u \in \mathcal{L}_A\} .\]

\end{defn}

\begin{lem}
 Let \(A,B \subset X\). Then \(\normcap(A \cup B) \leq \normcap(A)+\normcap(B)\) and if \(A\subset B\), then \(\normcap(A) \leq \normcap(B)\).
\end{lem}
\begin{proof}
For the first part, let us consider two arbitrary functions \(f_A\in \mathcal{L}_A,f_B\in \mathcal{L}_B\). Then
\begin{align*}
    \normcap(A \cup B) \leq \Vert f_A \vee f_B \Vert_\D \leq \Vert f_A  \Vert_\D+\Vert  f_B \Vert_\D.
\end{align*}
Since \(f_A,f_B\) are arbitrary, this implies the first claim.
 The second part follows directly from the definition and properties of the infimum.

\end{proof}
\begin{lem}
Let \(A_n ,A \subset X\) such that \(\bigcup_{n=1}^\infty A_n =A\). Then \(\sum_{n =1}^\infty \normcap(A_n)\geq \normcap(A)\).
\end{lem}
 \begin{proof}

 Let \(A_n,A \subset X\) as above. 
 
Let us assume \(\sum_{n =1}^\infty \normcap(A_n) < \infty\). Otherwise there is nothing to show. Let \( \epsilon>0\).
 By the properties of the infimum, we can choose functions \(u_n \in \lspace\) and open sets \(U_n\) for every \(n \in \mathbb{N}\) such that \(u_n \geq 1\) on \(U_n\) and \(A_n \subset U_n\)
\[\Vert u_n \Vert_\D \leq \normcap(A_n) + \epsilon 2^{-n}.\]
Without loss of generality we assume \(u_n\leq 1 \), otherwise take \(u_n \wedge 1\). Set
\begin{align*}
    U= \bigcup_{n=1}^\infty U_n
\end{align*}
and
\begin{align*}
    g_n=\bigvee_{i=1}^n u_i.
\end{align*}
By \cref{Thm:D_Riesz_Subspace_Norm},
\begin{align*}
    \Vert g_n \Vert_\D \leq \sum_{n =1}^\infty \normcap(A_n) +\epsilon < \infty.
\end{align*}
Since \(\Vert \cdot \Vert_{\lspace} \leq \Vert \cdot \Vert_\D\), the sequence \(g_n\) is bounded in \(L^2\). Thus, there exists a weakly convergent subsequence, again denoted by \(g_n\), such that \(g_n \rightharpoonup g\) for some \(g \in \lspace\). 
For every \(k\in \mathbb{N}\) we have \(g_n \rightharpoonup g\) in \( L^2(\bigcup_{j=1}^k U_j,m)\), but \(g_i=1\) in \( L^2(\bigcup_{j=1}^k U_j,m)\) for every \(i>k\). Hence, \(g=1\) on \(U_k\) for every \(k \in \mathbb{N}\). Thus, \(g=1\) on \(U\) and, by the weak lower semicontinuity of \(\Vert \cdot \Vert_\D\) on \(\lspace\), we have
\begin{align*}
 \normcap(A)&\leq  \Vert g \Vert_\D \\ &\leq \liminf_{n \rightarrow \infty} \Vert g_n \Vert_\D \\ &\leq  \sum_{n=1}^\infty \normcap(A_n) + \epsilon . 
\end{align*}
Since \(\epsilon\) was arbitrary, this concludes the proof.

\end{proof}

\begin{lem}
Let \(K_n \subset X\) be compact subsets such that \(K_n \downarrow K \). Then
\begin{align*}
    \normcap(K)= \inf_{n \in \mathbb{N}} \normcap(K_n).
\end{align*}
\end{lem}
\begin{proof}
By the monotonicity of \( \normcap\), \(\normcap(K)\leq \inf_{n \in \mathbb{N}} \normcap(K_n) \).

For the converse inequality, let us choose an \(\epsilon>0\), an open neighborhood \(U\) of \(K\) and a function \(f \in \mathcal{L}_U\) such that
\begin{align*}
    \Vert f\Vert_\D \leq \normcap(K) +\epsilon.
\end{align*}

We know
\begin{align*}
    \emptyset = K \cap U^c = \bigcap_{i=1}^\infty K_i \cap (U^c\cap K_1).
\end{align*}
For every \(i \in \mathbb{N}\), \(K_i \subset K_1\) and \(U^c \cap K_1 \subset K_1\). Since \(K_1\) is compact, we can apply the finite intersection property. Hence, there are finitely many \(K_{i_1},\dots, K_{i_n}\) such that 
\begin{align*}
\emptyset  = \bigcap_{k=1}^n K_{i_k} \cap (U^c\cap K_1)=K_{i_n} \cap (U^c\cap K_1) = K_{i_n} \cap U^c.    
\end{align*}
Thus, \(U\) is an open neighborhood of \(K_{i_n}\) and, by the definition of the capacity, 
\begin{align*}
    \inf_{n \in \mathbb{N}} \normcap(K_n) \leq \normcap(K_{i_n}) \leq \Vert f \Vert_\D \leq \normcap(K) +\epsilon.
\end{align*}
Since \(\epsilon\) is arbitrary, the claim follows.
\end{proof}

\begin{defn}
We call a set \(A\subset X\) \underline{polar}, if \( \normcap(A)=0\), and some property holds \underline{quasi everywhere} (q.e.), if it holds up to a polar set.
\end{defn}

\begin{exmp}
    Let \(X \subset \mathbb{R}^n\) a domain, \(m\) the Lebesgue measure and \(p \geq 1\). Then,
    \begin{align*}
        \E(u)= \frac{1}{p} \int_X \vert \nabla u \vert ^p
    \end{align*}
    is a convex, lower semicontinuous Dirichlet form. The Dirichlet space \(\D= W^1_{p,2}(X)=\{u \in \lspace \setdelim \nabla u \in L^p\}\) and the capacity of \(\E\) is the relative \(p\)-capacity (or the usual \(p\)-capacity if \(X=\mathbb{R}^n\)). Let \(\alpha \geq 2\). The perturbation
    \begin{align*}
        \E_\alpha (u)= \E(u) + \int_X \vert u \vert^ \alpha
    \end{align*}
    is also a convex, lower semicontinuous Dirichlet form. 
    
    Let \(A \subset X\) be a \(\normcap_{\E}\)-polar set. Then there is a sequence \((u_n)_{n\in \mathbb{N}}\) in \(\D\) such that \(u_n =1 \) on a neighbourhood of \(A\) and \(u_n \rightarrow 0\) in \(\D\). Hence, \(u_n \rightarrow 0\) in \(\lspace\). By the Hölder inequality \(u_n \rightarrow 0 \) in \(L^\alpha(X)\). Therefore \(A\) is \(\normcap_{\E_\alpha}\)-polar. The converse implication, namely that every \(\normcap_{\E_\alpha}\)-polar set is \(\normcap_{\E}\)-polar,  is clear. Hence, the capacities possess the same polar sets. Note carefully that the corresponding Dirichlet spaces need not be the same.
\end{exmp}

\section{Quasicontinuity}
In this section \(\E\) denotes a symmetric Dirichlet form.
\begin{defn}
 We say \(f: X  \rightarrow Y\) for some topological space \(Y\) is \underline{quasicontinuous}, if for every \(\epsilon > 0 \) there is an open set \(O\subset X \) such that \(\normcap(O) \leq \epsilon \) and \(f|_{O^c}\) is continuous.

Additionally we call \(f \in \lspace\) quasicontinuous, if there is a representative which is quasicontinuous. Whenever this is the case, we denote this representative again by \(f\).
\end{defn}

\begin{lem}\label{lem:polar_implies_nullset}
 Let \( A \subset X\) be measurable and polar. Then \(m(A)=0\).
\end{lem}
\begin{proof}
 Let \( A \subset X\) be measurable and polar. Then there exists a sequence of open sets \(A \subset O_n \subset X\) and functions \(f_n\geq 0\) such that \(f_n\geq 1\) on \(O_n\) and \(  \Vert f_n\Vert_\D \leq \frac{1}{n}\). Hence,
 \begin{align*}
      m(A)=\Vert 1_A \Vert_2^2 \leq \Vert f_n \Vert_2^2 \leq \Vert f_n \Vert_\D^2 \leq \frac{1}{n^2},
 \end{align*}
 which implies \(m(A)=0\).
\end{proof}

\begin{thm} \label{aeqe}
Let \(U \subset X\) be an open subset and \(f:X \rightarrow \mathbb{R} \) quasicontinuous. Then
\begin{align*}
f \geq 0 \text{ a.e. on }U \Longleftrightarrow f \geq 0 \text{ q.e. on }U
\end{align*}
\end{thm}
\begin{proof}
Let us assume \(f \geq 0 \text{ a.e. on U}\). We choose an \(\epsilon>0\) and an open set \(O \subset X\) as in the definition of quasicontinuity. We know that \(N= \{f < 0 \} \cap U\) has measure \(0\). First observe that f is continuous on \(O^c\) and thus \(O^c \cap \{f<0\}\) is open in \(O^c\). This shows that 
\begin{align*}
    O'=O \cup N= O \cup \left( (N \cap O )\cup (N\cap O^c)\right)= O \cup (N\cap O^c)
\end{align*}  is open in \(X\) by the definition of the induced topology. Thus, \(\normcap(O') = \normcap(O)\), since \(\mathcal{L}_O=\mathcal{L}_{O'}\). Therefore
\begin{align*}
\normcap(N) \leq \normcap(O') = \normcap(O) < \epsilon.
\end{align*}
Since \(\epsilon\) was arbitrary we have \(\normcap(N)=0\). The other direction follows directly from \cref{lem:polar_implies_nullset}.

\end{proof}

\begin{corollary}\label{thm:unique_qc_representatives}
Let \(f_1 , f_2:X \rightarrow \mathbb{R}\) be two quasicontinuous representatives of some \(f \in \lspace\). Then \(f_1=f_2\) quasi everywhere on \(X\).
\end{corollary}

\begin{defn}
We call a set \(U \subset X\) \underline{quasi open} if, for every \(\epsilon>0\) there is an open set \(O_\epsilon\subset X\) such that \(U \subset O_\epsilon\) and \(\normcap(O_\epsilon \setminus U) \leq \epsilon\).
\end{defn}

\begin{lem}
A function \(f \in \lspace\) is quasicontinuous if and only if, for every open set \(U \subset \mathbb{R}\), \(f^{-1}(U)\) is quasi open.
\end{lem}
\begin{proof}
The proof works exactly like the one in the bilinear case \cite{FOT_Dirichlet_forms_Markov_Processes}.
\end{proof}

\begin{thm} \label{thm:cheb_inequality}
Let \(f \in \mathfrak{D}\) be a quasicontinuous function and \(\lambda>0\). Then 
\begin{align*}
\normcap(\{|f| > \lambda\}) \leq \lambda^{-1} \Vert f \Vert_\D.
\end{align*}
\end{thm}

\begin{proof}
Let \(\lambda>0\) and \(\epsilon > 0 \). Since \(f\) is quasicontinuous, there is an open set \(U_\epsilon \) with \( \normcap(U_\epsilon) \leq \epsilon \) and \(f\) is continuous on \(U_\epsilon^c\). There is a function \(g_\epsilon\), such that \(g_\epsilon\geq 1\) on \(U_\epsilon\) and \(\Vert g_\epsilon\Vert_\D \leq 2 \epsilon \). Additionally, the set \( O= \{ |f|> \lambda \} \cup U_\epsilon \) is open. Note that \((\lambda^{-1} f)\vee g_\epsilon \geq 1\) on \(O\). Hence,
\begin{align*}
\normcap(\{|f| > \lambda\}) &\leq \normcap(O) \leq \Vert(\lambda^{-1} f)\vee g_\epsilon\Vert_\D\\ &\leq  \Vert \lambda^{-1} f\Vert_\D + \Vert g_\epsilon\Vert_\D \leq \lambda^{-1} \Vert f\Vert_\D + 2\epsilon.
\end{align*}
Since \(\epsilon\) is arbitrary, this implies the claim.
\end{proof}

\begin{thm}\label{thm:pointwise_quasi_everywhere}
Let \((f_n)_n\) be a sequence of quasicontinuous functions in \(\mathfrak{D}\) and \(f \in \mathfrak{D}\) with \(f_n \rightarrow f \) in \( \mathfrak{D}\). Then \(f\) is quasicontinuous and there exists a subsequence which converges pointwise quasi everywhere and quasi uniformly, that is for every \(\epsilon >0\) there is a open set \(U \subset X\) such that \(\normcap(U)\leq \epsilon\) and \(f_n\) converges uniformly to \(f\) on \(U^c\).
\end{thm}
\begin{proof}
Let us choose a subsequence \((f_n)_n\) such that
\begin{align*}
\Vert \vert f_n-f_{n+1}\vert \Vert_\D= \Vert f_n-f_{n+1} \Vert_\D \leq 2^{-2 n}.
\end{align*}
Then, since \(|f_n-f_{n+1}|\) is quasicontinuous, \cref{thm:cheb_inequality} yields
\begin{align*}
\normcap(\{|f_n-f_{n+1}|>2^{-n}\}) \leq 2^{-n}.
\end{align*}
Hence, for every \(m \in \mathbb{N}\) we have
\begin{align*}
\normcap(\bigcup_{n \geq m} \{|f_n-f_{n+1}|>2^{-n}\} ) \leq 2^{-m+1}.
\end{align*}
In addition, let \(\epsilon>0\). Since \(f_n\) is quasicontinuous, we can choose an open set \(U_n\) such that
\begin{align*}
    \normcap(U_n) \leq \epsilon 2^{-n}
\end{align*}
and \(f_n\vert_{U_n^c}\) is continuous. Hence,
\begin{align*}
\normcap(U)=\normcap(\bigcup U_n) \leq \epsilon.
\end{align*}
Therefore, the sequence \((f_n\vert_{U^c})_{n \in \mathbb{N}}\) of continuous functions converges uniformly on \(A_m= \bigcap_{n \geq m} \{|f_n-f|\leq 2^{-n}\}\). Thus, \(f\vert_{U^c \cap A_m}\) is continuous, since it is the uniform limit of continuous functions, and 
\begin{align*}
    \normcap\big((A_m\cap U^c)^c\big) &\leq \normcap\left(\bigcup_{n \geq m} \{|f_n-f_{n+1}|>2^{-n}\}\right)+\normcap(U) \\
    &\leq 2^{m-1}+\epsilon.
\end{align*}
This implies, that \(f\) is quasicontinuous.
Next, we show that a subsequence of \(f_n\) converges pointwise quasi everywhere. For this purpose, let us choose a different subsequence such that
\begin{align*}
\Vert \vert f_n-f\vert \Vert_\D= \Vert f_n-f \Vert_\D \leq 2^{-2 n}.
\end{align*}

Let \(m \in \mathbb{N}\) be arbitrary and \(x \in \bigcap_{n \geq m} \{|f_n-f|\leq 2^{-n}\}\). Then \(f_n(x) \rightarrow f(x)\). Thus, \(A=\{x\setdelim f_n(x) \nrightarrow f(x)\} \) is a subset of \(\bigcup_{n \geq m} \{|f_n-f|>2^{-n}\}\) for every \(m\). The Chebychev type inequality \cref{thm:cheb_inequality} yields
\begin{align*}
\normcap(A)\leq \normcap(\bigcup_{n \geq m} \{|f_n-f|>2^{-n}\} ) \leq 2^{-m+1}
\end{align*}
for every \(m\). Therefore 
\begin{align*}
\normcap(A)=0.
\end{align*}
\end{proof}

\begin{corollary}
Let \(f \in \overbar{\mathfrak{D} \cap C(X)}^\D\). Then \(f\) is quasicontinuous on \(X\).
\end{corollary}

\begin{proof}
Every continuous function is quasicontinuous. Hence, the previous theorem implies the claim.
\end{proof}

\begin{remark}
For \(\E=\Vert \cdot \Vert_2\) this theorem is a version of Lusin's Theorem and Egorov's theorem. Furthermore, \cref{thm:cheb_inequality} is a version of the Markov or weak \(L^1\) inequality.
\end{remark}

\bibliography{phdbib}
\bibliographystyle{alpha}
\end{document}